\documentclass[12pt, reqno]{amsart}
\usepackage{amsmath, amsthm, amscd, amsfonts, amssymb, graphicx, color, mathrsfs}
\usepackage[bookmarksnumbered, colorlinks, plainpages]{hyperref}
\usepackage[all]{xy} 
\usepackage{slashed}

\newtheorem{theorem}{Theorem}[section]

\newtheorem{corollary}[theorem]{Corollary}

\theoremstyle{definition}
\newtheorem{definition}[theorem]{Definition}

\theoremstyle{remark}
\newtheorem{remark}[theorem]{Remark}
\numberwithin{equation}{section}

\begin{document}
\setcounter{page}{1}

\title[Sharpness  of Seeger-Sogge-Stein orders]{ Sharpness  of Seeger-Sogge-Stein orders for the weak (1,1) boundedness of Fourier integral operators}

\author[D. Cardona ]{Duv\'an Cardona}
\address{
  Duv\'an Cardona:
  \endgraf
  Department of Mathematics: Analysis, Logic and Discrete Mathematics
  \endgraf
  Ghent University, Belgium
  \endgraf
  {\it E-mail address} {\rm duvanc306@gmail.com, duvan.cardonasanchez@ugent.be}
  }

\author[M. Ruzhansky]{Michael Ruzhansky}
\address{
  Michael Ruzhansky:
  \endgraf
  Department of Mathematics: Analysis, Logic and Discrete Mathematics
  \endgraf
  Ghent University, Belgium
  \endgraf
 and
  \endgraf
  School of Mathematical Sciences
  \endgraf
  Queen Mary University of London
  \endgraf
  United Kingdom
  \endgraf
  {\it E-mail address} {\rm michael.ruzhansky@ugent.be, m.ruzhansky@qmul.ac.uk}
  }

\thanks{The authors are supported  by the FWO  Odysseus  1  grant  G.0H94.18N:  Analysis  and  Partial Differential Equations and by the Methusalem programme of the Ghent University Special Research Fund (BOF)
(Grant number 01M01021). Michael Ruzhansky is also supported  by EPSRC grant 
EP/R003025/2.
}


\begin{abstract} Let  $X$ and $Y$ be two smooth manifolds  of the same dimension. It was proved by Seeger, Sogge and Stein in \cite{SSS} that the   Fourier integral operators with real non-degenerate phase functions in the class $I^{\mu}_1(X,Y;\Lambda),$ $\mu\leq -(n-1)/2,$ are bounded from $H^1$ to $L^1.$ The sharpness of the order $-(n-1)/2,$ for any elliptic operator was also proved in \cite{SSS} and extended to other types of canonical relations in \cite{Ruzhansky1999}. That the operators in the class  $I^{\mu}_1(X,Y;\Lambda),$ $\mu\leq -(n-1)/2,$  satisfy the weak (1,1) inequality   was proved by Tao \cite{Tao:weak11}. In this note, we prove that the weak (1,1) inequality for the order $ -(n-1)/2$  is sharp  for any elliptic Fourier integral operator, as well as its versions for canonical relations satisfying additional rank conditions. 
\end{abstract} 

\maketitle
\tableofcontents

\allowdisplaybreaks

\section{Introduction} 

Let $X$ and $Y$ be  smooth manifolds of dimension $n.$ In this work, we analyse the sharpness of the order $-(n-1)/2$  for the weak (1,1) inequality of elliptic Fourier integral operators  $T\in I^\mu_1(X,Y;\Lambda)$ with order $\mu\leq -(n-1)/2.$ For the general aspects of the theory of Fourier integral operators we refer the reader to H\"ormander  \cite{Ho}, Duistermaat and H\"ormander \cite{Duistermaat-Hormander:FIOs-2} and  Melin and Sj\"ostrand \cite{Melin:SjostrandI,Melin:SjostrandII}.

First, let us review the mapping properties of Fourier integral operators.  By following \cite{Duistermaat-Hormander:FIOs-2,Ho}, these classes are denoted by $I^\mu_\rho(X,Y;\Lambda),$ with $\Lambda$ being  considered locally a graph of a symplectomorphism from $T^{*}X\setminus 0$ to  $T^{*}Y\setminus 0,$ which are equipped with the canonical symplectic forms $d\sigma_X$ and $d\sigma_Y,$ respectively. Such Fourier integral operators are called non-degenerate.  The symplectic structure of $\Lambda$ is determined by the symplectic $2$-form $\omega$ on $X\times Y,$ $\omega=\sigma_X\oplus -\sigma_Y .$ Let $\pi_{X\times Y}$ be the canonical projection from $T^*X\times T^*Y$ into $X\times Y.$ As in the case of pseudo-differential operators, non-degenerate Fourier integral operators of order zero are bounded on $L^2.$ The fundamental work of Segger, Sogge and Stein \cite{SSS} establishes the boundedness of 
 $T\in I^\mu_1(X,Y;\Lambda),$ from $L^p_{\textnormal{comp}}(Y)$ into $L^p_{\textnormal{loc}}(X)$ with the order
 \begin{equation}\label{critical:SSS}
     \mu\leq -(n-1)|1/2-1/p|,
 \end{equation}if $1<p<\infty,$ and from $H^1_{\textnormal{comp}}(Y)$ into $L^1_{\textnormal{loc}}(X)$ if $p=1.$ Also, for $p=1$ in \eqref{critical:SSS} and as a consequence of the weak (1,1) estimate in  Tao \cite{Tao:weak11}, an operator $T$ of order $-(n-1)/2$ is locally of weak (1,1) type.
 
The critical Seeger-Sogge-Stein order \eqref{critical:SSS} is sharp if $d\pi_{X\times Y}|_{\Lambda}$ has full rank equal to $2n-1$ somewhere and if $T$ is an elliptic operator. When  $d\pi_{X\times Y}|_{\Lambda}$ does not attain the maximal rank $2n-1,$ the upper bound for the order \eqref{critical:SSS} is not sharp and may depend of the singularities of $d\pi_{X\times Y}|_{\Lambda}.$  In conclusion, as it was observed in  \cite{SSS}, the mapping properties of the classes $ I^\mu_\rho(X,Y;\Lambda)$ of Fourier integral operators depend of the singularities and of the maximal rank of the canonical projection. So, an additional condition on the canonical relation $\Lambda$ was introduced in  \cite{SSS}, the so called, factorisation condition for $\pi_{X\times Y}$. Roughly speaking, it can be introduced as follows. Assume that there exists $k\in \mathbb{N},$ with $0\leq k\leq n-1,$ such that for any $\lambda_0=(x_0,\xi_0,y_0,\eta_0)\in \Lambda,$ there is a conic neighborhood $U_{\lambda_0}\subset \Lambda$ of $\lambda_0,$ and a smooth homogeneous of order zero map $\pi_{\lambda_0}:U_{\lambda_0}\rightarrow \Lambda,$ such that 
 \begin{equation}\label{FC:Intro:1}
    \textnormal{(RFC):   } \textnormal{rank}(d\pi_{\lambda_0})=n+k,\textnormal{   and   }\pi_{X\times Y}|_{U_{\lambda_0}}=\pi_{X\times Y}|_{\Lambda}\circ \pi_{\lambda_0}.
\end{equation}Under the real factorisation condition (RFC) in  \eqref{FC:Intro:1},  Seeger, Sogge and Stein in   \cite{SSS} proved that the order
 \begin{equation}\label{rho=1}
     \mu\leq -(k+(n-k)(1-\rho))|1/2-1/p|,
\end{equation}guarantees that any Fourier integral operator $T\in I^\mu_\rho(X,Y;\Lambda),$ with $\rho\in [1/2,1],$ is bounded from $L^p_{\textnormal{comp}}(Y)$ into $L^p_{\textnormal{loc}}(X),$ and for $p=1,$ from $H^1_{\textnormal{comp}}(Y)$ into $L^1_{\textnormal{loc}}(X).$  For a set $A,$ $1_A$ denotes its characteristic function. We recall that $T$    is locally  of weak (1,1) type if, for any pair of compact subsets $K\subset Y,$ and $K'\subset X,$ the localised operator $1_{K'}T1_K:L^{1}(Y)\rightarrow L^{1,\infty}(X)$ is bounded.

 If $\Lambda$ satisfies the factorisation condition (RFC) in \eqref{FC:Intro:1}, because $\textnormal{rank}(d\pi_{\lambda_0})=n+k,$ we have that
 $
    \textnormal{rank}(d{\pi_{X\times Y}}|_{U_{\lambda_0}})\leq n+k,
$
 in an open neighborhood $U_{\lambda_0}$ of $\lambda_0.$ The following Theorem \ref{sharpness} proves that if the rank  $n+k$ is attained somewhere, the order in \eqref{rho=1} for $\rho=1,$ that is $\mu\leq-k/2$ is sharp, for $0\leq k\leq n-1$. The sharpness of the order $\mu\leq-k/2$ for the $H^1_{\textnormal{loc}}(Y)$-$L^{1}_{\textnormal{loc}}(X)$-boundedness of elliptic Fourier integral operators has been proved in  \cite{Ruzhansky1999}, with the case $k=n-1$ known from  \cite{SSS}. Here, we observe that the geometric construction in \cite{Ruzhansky1999} implies also the following sharp Theorem \ref{sharpness}, proving for $k=n-1,$ the converse of the weak (1,1)-inequality in Tao \cite{Tao:weak11} for elliptic Fourier integral operators. More specifically, we have:

\begin{theorem}\label{sharpness}
Let the real canonical relation $\Lambda$ be a local canonical graph such that the inequality $\textnormal{rank}(d{\pi_{X\times Y}}|_{\Lambda})\leq n+k,$ holds with $0\leq k\leq n-1,$ and the rank $n+k$ is attained at some point. Then elliptic operators $T\in I_{1}^{\mu}(X,Y;\Lambda)$ are not locally of weak (1,1) type provided that $\mu>-k/2.$
\end{theorem}
\begin{remark}
 In the endpoint case $k=n-1,$ $\textnormal{rank}(d{\pi_{X\times Y}}|_{\Lambda})\leq 2n-1,$ the elliptic operators $T\in I_{1}^{\mu}(X,Y;\Lambda)$ are not locally of weak (1,1) type provided that $\mu>-(n-1)/2.$ In view of the weak (1,1) estimate in Tao \cite{Tao:weak11} for the class $I_{1}^{-(n-1)/2}(X,Y;\Lambda),$ when the real canonical  relation has full rank  $\textnormal{rank}(d{\pi_{X\times Y}}|_{\Lambda})\leq 2n-1,$ the main Theorem 1.1  in  \cite{Tao:weak11} together with Theorem \ref{sharpness} imply the following result. 
\end{remark}
\begin{corollary} Let the real canonical relation $\Lambda$ be a local canonical graph such that (RFC) in \eqref{FC:Intro:1} is satisfied for $k=n-1.$ Let $T\in I_{1}^{\mu}(X,Y;\Lambda)$ be an elliptic Fourier integral operator. Then, $T$ is locally of weak (1,1) 
type if and only if, $\mu\leq -(n-1)/2.$
\end{corollary}

\section{Preliminaries}

\subsection{Basics on symplectic geometry}
Let $M,$ $X$ and $Y$ be (paracompact) smooth real manifolds of dimension $n.$ So, using partitions of the unity the spaces $L^{1}(Y)$ and $L^{1,\infty}(X)$ are defined by the set of functions that under any  changes of coordinates belong to $L^{1}(\mathbb{R}^n)$ and  $L^{1,\infty}(\mathbb{R}^n),$ respectively. For instance, we can take $M=X,Y$ or $M=X\times Y.$ In this section we will follow  \cite[Chapter I]{Ruzhansky:CWI-book}.

A $2$-form $\omega$ is called {\it symplectic}  on $M$ if $d\omega =0,$ and for all $x\in M,$ the bilinear form $\omega_x$  is antisymmetric and non-degenerate on $T_{x}M$. The canonical symplectic form $\sigma_M$ on $M$ is defined as follows. Let $\pi:=\pi_M:T^*M\rightarrow M,$ be the canonical projection. For any $(x,\xi)\in T^*M,$ let us consider the linear mappings
\begin{equation*}
    d\pi_{(x,\xi)}:T_{(x,\xi)}(T^*M)\rightarrow T_{x}M\textnormal{   and,   }\xi:T_{x}M\rightarrow \mathbb{R}.
\end{equation*}The composition $\alpha_{(x,\xi)}:=\xi\circ d\pi_{(x,\xi)}\in T^{*}_{(x,\xi)}(T^*M),$ that is
$
   \xi\circ d\pi_{(x,\xi)}:T_{(x,\xi)}(T^*M)\rightarrow  \mathbb{R},
$ defines  a 1-form $\alpha $ on $T^*M.$ Then, the canonical symplectic form $\sigma_M$ on $M$ is defined by 
\begin{equation}\label{sigmaM}
    \sigma_M:=d\alpha.
\end{equation}Because $\sigma_M$ is an exact form, it follows that $d\sigma_M=0$ and then  that $\sigma_M$ is symplectic. If $M=X\times Y,$ it follows that  $\sigma_{X\times Y}=\sigma_X\oplus -\sigma_Y .$ Now we record the  kind of submanifolds  that are necessary when one defines the canonical relations.
\begin{itemize}
    \item A submanifold $\Lambda\subset T^*M$ of dimension $n$ is called {\it Lagrangian} if 
\begin{equation*}
    T_{(x,\xi)}\Lambda=( T_{(x,\xi)}\Lambda)^{\sigma}:=\{y\in T_{(x,\xi)}(T^*M):\sigma_M(y,y')=0,\,\forall y'\in T_{(x,\xi)}\Lambda \}.
\end{equation*}
\item We say that  $\Lambda\subset T^*M\setminus 0$ is {\it conic} if $(x,\xi)\in \Gamma,$ implies that $(x,t\xi)\in \Gamma,$ for all $t>0.$ 
\item Let $\Sigma\subset X$ be a smooth submanifold of $X$ of dimension $k.$ Its conormal bundle in $T^*X$ is defined by
\begin{equation}
    N^{*}\Sigma:=\{(x,\xi)\in T^*X: \,x\in \Sigma,\,\,\xi(\delta)=0,\,\forall \delta\in T_{x}\Sigma  \}.
\end{equation}
\end{itemize}
The following facts characterise the Lagrangian submanifolds of $T^*M.$
\begin{itemize}
    \item Let $\Lambda\subset T^*M\setminus 0,$ be a closed sub-manifold of dimension $n.$ Then $\Lambda$ is a conic Lagrangian manifold if and only if the 1-form $\alpha$ in \eqref{sigmaM} vanishes on $\Lambda.$
    \item Let $\Sigma\subset X,$ be a submanifold of dimension $k.$ Then its conormal bundle $N^*\Sigma$ is a conic Lagrangian manifold. 
\end{itemize}
The Lagrangian manifolds have the following property.
\begin{itemize}
    \item Let $\Lambda\subset T^{*}M\setminus 0,$ be a conic Lagrangian manifold and let 
    \begin{equation}
        d\pi_{(x,\xi)}:T_{(x,\xi)}\Lambda\rightarrow T_{x}M,
    \end{equation}have constant rank equal to $k,$ for all $(x,\xi)\in \Lambda.$ Then, each $(x,\xi)\in \Gamma$ has a conic neighborhood $\Gamma$ such that 
    \begin{itemize}
        \item[1.] $\Sigma=\pi(\Gamma\cap \Lambda)$ is a smooth manifold of dimension $k.$
        \item[2.] $\Gamma\cap \Lambda$ is an open subset of $N^*\Sigma.$
    \end{itemize}
\end{itemize}
The Lagrangian manifolds have a local representation defined in terms of {\it phase functions} that can be defined as follows. For this, let us consider a local trivialisation $M\times (\mathbb{R}^n\setminus 0), $ where we can assume that $M$ is an open subset of $\mathbb{R}^n.$

\begin{definition}[Real-valued phase functions]Let $\Gamma$ be a cone in $M\times (\mathbb{R}^N\setminus 0). $ A smooth function $\phi:M\times (\mathbb{R}^N\setminus 0)\rightarrow \mathbb{R}, $ $(x,\theta)\mapsto \phi(x,\theta),$ is a real {\it phase function} if, it is homogeneous of degree one in $\theta$ and has no critical points as a function of $(x,\theta),$ that is 
\begin{equation}
 \forall t>0,\,   \phi(x,t\theta)=t\phi(x,\theta), \textnormal{   and   }d_{(x,\theta)}\phi(x,\theta)\neq 0,\forall (x,\theta)\in M\times (\mathbb{R}^N\setminus 0).
\end{equation}Additionally, we say that $\phi$ is non-degenerate in $\Gamma,$ if for $(x,\theta)\in \Gamma$ such that $d_{\theta}\phi(x,\theta)=0,$ one has that 
\begin{equation}
    d_{(x,\theta)}\frac{\partial \phi}{\partial\theta_{j}}(x,\theta),\,\,1\leq j\leq N,
\end{equation}is a system of linearly independent vectors.
\end{definition}
The following facts describe locally a Lagrangian manifold in terms of a phase function.
\begin{itemize}
    \item Let $\Gamma$ be a cone in $M\times (\mathbb{R}^N\setminus 0), $ and let $\phi$ be a non-degenerate phase function in $\Gamma.$ Then, there exists an open cone $\tilde{\Gamma}$ containing $\Gamma$ such that the set
    \begin{equation}
        U_{\phi}=\{(x,\theta)\in \tilde{\Gamma}:d_\theta \phi(x,\theta)=0\},
    \end{equation}is a smooth conic  sub-manifold of $M\times (\mathbb{R}^N\setminus 0)$ of dimension $n.$ The mapping 
    \begin{equation}
        L_{\phi}: U_\phi\rightarrow T^*M\setminus 0,\,\, L_{\phi}(x,\theta)=(x,d_x\phi(x,\theta)),
    \end{equation}is an inmersion. Let us denote $\Lambda_\phi=L_{\phi}(U_\phi).$
    \item Let $\Lambda\subset T^*M\setminus 0$ be a sub-manifold of dimension $n.$ Then,  $\Lambda$ is a conical Lagrangian manifold if and only if  every $(x,\xi)\in \Lambda$ has a conic neighborhood $\Gamma$ such that $\Gamma\cap \Lambda=\Lambda_\phi, $ for some non-degenerate phase function $\phi.$ 
\end{itemize}
\begin{remark}
The cone condition on $\Lambda$ corresponds to the homogeneity of the phase function.
\end{remark}
\begin{remark} Although we have given definition  when a real  phase function  of  $(x,\theta)$ is non-degenerate, the same can be defined if one considers functions of $(x,y,\theta).$ Indeed, a real valued phase function $\phi(x,y,\theta)$ homogeneous of order 1 at $\theta\neq 0$  that satisfies the following two conditions
 \begin{equation}
        \textnormal{det}\partial_x \partial_\theta(\phi(x,y,\theta))\neq 0,\,\, \textnormal{det}\partial_y \partial_\theta(\phi(x,y,\theta))\neq 0,\,\,\theta\neq 0,
    \end{equation} is called  non-degenerate. 

\end{remark}

\subsection{Smooth factorisation condition for real phases} We can assume that  $X, Y$ are open sets in $\mathbb{R}^n$. One defines the class of
  Fourier integral operators $T\in I^\mu_{\rho}(X, Y;\Lambda)$ by
  the (microlocal) formula
   \begin{equation}\label{EQ:FIO}
     Tf(x)=\int\limits_Y\int\limits_{\mathbb{R}^N}
   e^{i\Psi(x,y,\theta)} a(x,y,\theta) f(y)d\theta\;dy,
   \end{equation} 
   where the symbol $a$ is a smooth function locally  in the class $S^\mu_{\rho,1-\rho}(X\times Y\times  (\mathbb{R}^n\setminus 0) ),$ with $1/2\leq \rho\leq 1.$ This means that  $a$ satisfies the symbol inequalities
   $$ |\partial_{x,y}^{\beta}\partial_\theta^\alpha
   a(x,y,\theta)|
    \leq C_{\alpha\beta}(1+|\theta|)^{\mu-\rho|\alpha|+(1-\rho)|\beta|},$$
   for $(x,y)$ in any compact subset $K$ of $X\times Y,$ and $\theta \in \mathbb{R}^N\setminus 0,$ while the real-valued phase function $\Psi$ satisfies 
     the following properties:
\begin{itemize}
    \item[1.] $\Psi(x,y,\lambda\theta)=
      \lambda\Psi(x,y,\theta),$ for all $\lambda>0$;
    \item[2.] $d\Psi\not=0$;
    \item[3.] $\{d_\theta\Psi=0\}$ is smooth 
      (i.e. $d_\theta\Psi=0$ implies
      $d_{(x,y,\theta)}\frac{\partial\Psi}{\partial\theta_j}$ are linearly
      independent).
\end{itemize}
Here $\Lambda \subset T^*(X\times Y)\setminus 0$ is a Lagrangian manifold locally parametrised by the phase function $\Psi,$
 $$ \Lambda=\Lambda_\Psi=\{(x,d_x\Psi,y,d_y\Psi): d_\theta\Psi=0\}.$$
The   canonical relation associated with $T$
  is the conic Lagrangian manifold in 
  $T^*(X\times Y)\backslash 0$,  defined by
  $\Lambda'=\{(x,\xi,y,-\eta): (x,\xi,y,\eta)\in \Lambda\}.$
  In view of the equivalence-of-phase-functions theorem (see e.g. Theorem 1.1.3 in \cite[Page 9]{Ruzhansky:CWI-book}), the notion of Fourier integral operator becomes independent of the choice of a particular phase function associated to a Lagrangian manifold $\Lambda.$  Because of the diffeomorphism $\Lambda\cong \Lambda',$ we do not distinguish between $\Lambda $ and $\Lambda'$ by saying also that $\Lambda$ is the canonical relation associated with $T.$ 

Let us consider the canonical projections:
    \begin{equation*}
    \begin{array}{ccccc}
       T^*X & \stackrel{\pi_X}{\longleftarrow} &
        \Lambda \subset T^*X\times T^*Y& 
        \stackrel{\pi_Y}{\longrightarrow} & T^*Y. \\
      & & \Big\downarrow\vcenter{%
         \rlap{$\pi_{X\times Y}$}}
      & & \\
      & & X\times Y & &
    \end{array}
   \label{eq:in-defpr}
  \end{equation*}
The smooth factorisation condition for $\Psi$ can be formulated as follows (see \cite{SSS} or e.g. \cite[Page 45]{Ruzhansky:CWI-book}). Suppose that there exists a number $k,$ with $0\leq k\leq n-1,$ such that for any $$\lambda_0=(x_0,\xi_0,y_0,\eta_0)\in \Lambda_\Psi, $$ there exists a conic neighborhood $U_{\lambda_0}\subset \Lambda_{\Psi},$ of $\lambda_0,$ and a smooth homogeneous of order zero map 
\begin{equation}
    \pi_{\lambda_0}:U_{\lambda_0}\rightarrow \Lambda_{\Psi},
\end{equation}with constant rank $\textnormal{rank}(d\pi_{\lambda_0})=n+k,$ for which one has
\begin{equation}\label{FCR}
   \textnormal{(RFC):  } \boxed{ \pi_{X\times Y}|_{U_{\lambda_0}}=\pi_{X\times Y}|_{\Lambda_{\Psi}}\circ   \pi_{\lambda_0}       }
\end{equation} 
In this case we say that the canonical relation $\Lambda:=\Lambda_\Psi$ satisfies the factorisation condition $\textnormal{(RFC)}.$

\section{Sharpness of Seeger-Sogge-Stein orders}
 Now, we will prove the sharpness of the Seeger-Sogge-Stein order for $\rho=1$ in the case of elliptic Fourier integral operators. 
\begin{proof}[Proof of Theorem \ref{sharpness}] Let $\mu>-k/2.$ Let us follow the  argument in \cite[Page 42]{Ruzhansky:CWI-book} (see also \cite{Ruzhansky1999}) and to analyse the weak (1,1) estimate at the end of the proof. Using again the equivalence-of-phase-function theorem, it is enough to consider elliptic operators $T$ on $\mathbb{R}^n$ with kernels, locally defined by
\begin{equation}
    K(x,y)=\int\limits_{\mathbb{R}^n}e^{i\Phi(x,y,\xi)}b(x,y,\xi)d\xi,\,\, \Phi(x,y,\xi):= x\cdot \xi-\phi(y,\xi),
\end{equation}with symbols $b(x,y,\xi)$ compactly supported in $(x,y).$ That $\Lambda$ satisfies the local graph condition means that the real-values phase function $\phi$ satisfies
\begin{equation}
    \textnormal{det}\partial_{y}\partial_{\xi}\phi(y,\xi)\neq 0
\end{equation} on the support of the symbol $b,$ and $\xi\neq 0.$ Let us observe that
\begin{equation}
    \Lambda_0:=\{\lambda\in \Lambda: \textnormal{rank }d\pi_{X\times Y}|_{\Lambda}(\lambda)=n+k\},
\end{equation}is not empty and is open in $\Lambda.$ Fix a point $\lambda_0\in \Lambda_0.$ Let $\Delta:=\sum_{j=1}^n\partial_{x_j}^2$ be the standard Laplacian on $\mathbb{R}^n,$ and define the distribution
\begin{equation}
    \varkappa(y):=(1-\Delta)^{-s/2}\delta_{y_0}, 
\end{equation}for a fixed point $y_0\in Y,$ where $s>0$ is small enough in such a way  that 
\begin{equation}\label{defi:of:s}
    -\frac{k}{2}+s<\mu.
\end{equation}
Since $(1-\Delta)^{-s/2}$ is an elliptic pseudo-differential operator of order $-s$, its Schwartz kernel $K_{s}$ is of Calder\'on-Zygmund type, and it satisfies the inequality  $|K_s(y,y_0)|\leq C |y-y_0|^{-n+s},$ in some local coordinate system. So, the fact that $s>0$ implies that
\begin{equation}
    \varkappa(y)=\int\limits_{\mathbb{R}^n}K_{s}(y,z)\delta_{y_0}(z)dz=K_{s}(y,y_0)\in L^{1}_{\textnormal{loc}}.
\end{equation}
Now, we are going to introduce  the geometric construction in  \cite{Ruzhansky1999} (see also \cite[Page 42]{Ruzhansky:CWI-book}) in order to obtain a useful parametrisation of the phase function $\Phi$.
Let $\Sigma=\pi_{X\times Y}(C\cap U),$ where $U\subset \Lambda_0$ is a neighborhood of $\lambda_0.$ Taking into account that $\textnormal{rank }d\pi_{X\times Y}|_{U}=n+k,$ that is the rank of $d\pi_{X\times Y}$ is constant in $U,$ $\Sigma $ is a $k$-dimensional sub-manifold defined by the equations
\begin{equation}
    h_{j}(x,y)=0,\,1\leq j\leq n-k,
\end{equation} in a neighborhood of $y_0,$ with the system $\{\nabla h_j:1\leq j\leq n-k\}$ being linearly independent on $\Sigma.$ Then $\Lambda$ is the conormal bundle of $\Sigma,$ and the phase function of $T$ takes the representation
\begin{equation}
    \Phi(x,y,\lambda)=\sum_{j=1}^{n-k}\lambda_jh_j(x,y).
\end{equation} Because compositions of Fourier integral operators with pseudo-differential operators leaves invariant the canonical relation $\Lambda$, we have that $T\circ (1-\Delta)^{-s/2}\in I^{\mu-s}_1(X,Y,\Lambda).$ So, in local coordinates we have
\begin{align*}
T\varkappa(x) &=T\circ (1-\Delta)^{-s/2}\delta_{y_0}=\int\limits_{\mathbb{R}^n}\int\limits_{\mathbb{R}^{n-k}}e^{i\left(\sum_{j=1}^{n-k}\lambda_jh_j(x,y)\right)}a(x,\overline{\lambda})\delta_{y_0}(y)d\overline{\lambda}dy\\
&=\int\limits_{\mathbb{R}^{n-k}}e^{i\overline{\lambda}\cdot \overline{h}(x,y_0)}a(x,\overline{\lambda})d\overline{\lambda}=(2\pi)^{n-k}(\mathscr{F}^{-1}a)(x,\overline{h}(x,y_0)),
\end{align*}where $\overline{\lambda}$ and $\overline{h}$ are vectors with components $\lambda_j$ and $h_j,$ respectively, and $\mathscr{F}$ denotes the  Fourier transform. The symbol $a\in S^{\mu-s+\frac{k}{2}}_{1,0}(\mathbb{R}^{n-k}),$ is obtained from the symbol of the operator $T\circ (1-\Delta)^{-s/2}$ by using the stationary method phase, where we have eliminated $k$-variables. Computing the second argument of $(2\pi)^{n-k}\mathscr{F}^{-1}a,$ one has
\begin{equation}\label{smoothidentity}
    (2\pi)^{n-k}\mathscr{F}^{-1}a(x,\zeta)=\int\limits_{\mathbb{R}^{n-k}}e^{i\lambda\cdot \zeta}a(x,\lambda)\mathscr{F}({\delta}_0)(\lambda)d\lambda=P\delta_{0}(\zeta),
\end{equation}where $P$ is a pseudo-differential operator in $\mathbb{R}^{n-k}$ of order $m=n-s+\frac{k}{2}.$ Denoting $K_P$ the Schwartz kernel of $P,$ we have that $P\delta_0(\zeta)=K_{P}(\zeta,0),$ and then one has 
\begin{equation}
    |K_{P}(\zeta,0)|\sim |\zeta|^{-(n-k)-m}, \,\,m=n-s+\frac{k}{2}.
\end{equation}
 Define the set $$\Sigma_{y_0}:=\{x:(x,y_0)\in \Sigma\}.$$
 We have that $\textnormal{distance}(x,\Sigma_{y_0})\asymp |\overline{h}(x,y_0)|.$ Consequently,
 \begin{align*}
    |(2\pi)^{n-k}\mathscr{F}^{-1}a(x,\zeta)|\asymp  \textnormal{distance}(x,\Sigma_{y_0})^{-(n-k)-(\mu-s+\frac{k}{2})},
 \end{align*}locally uniformly in $x.$ The identity \eqref{smoothidentity} implies that $T\varkappa$ is smooth on $\Sigma_{y_0}.$ Now, we apply the geometric construction above to test the operator $T$ into the distribution $\varkappa.$  We have proved the estimate
\begin{align*}
   | T\varkappa(x)| &=|K_{P}(\overline{h}(x,y_0),0)| \asymp |\overline{h}(x,y_0)|  \asymp \textnormal{distance}(x,\Sigma_{y_0})^{-(n-k)-(\mu-s+\frac{k}{2})},
   \end{align*} and that the singularities of $T\varkappa$ can  appear only in  transversal directions to $\Sigma.$ Now, to finish the proof,
let $\Omega=\overline{B(y_0,r)}$ be the compact neighborhood of $y_0,$ with radius $r>0.$ Observing that for any $x\in \Omega,$ $\overline{h}(x,y_0)\in \mathbb{R}^{n-k},$ $ T\varkappa \in L^{1}(\Omega),$  if and only if 
 $$ (n-k)+(\mu-s+\frac{k}{2})< n-k,    $$ 
 and that 
$T\varkappa\in L^{1,\infty}(\Omega)\setminus L^{1}(\Omega),$ 
  if and only if 
 $$ (n-k)+(\mu-s+\frac{k}{2})= n-k,  $$ 
it follows that $  T\varkappa \notin L^{1,\infty}(\Omega)$ if and only if 
 \begin{equation}
     (n-k)+(\mu-s+\frac{k}{2})>n-k,
 \end{equation}or equivalently $\mu>s-\frac{k}{2}$ as in \eqref{defi:of:s}. In conclusion we have that  $\varkappa$ is  locally in $L^1$ and that $T\varkappa \notin L^{1,\infty}(\Omega)$ showing that $T$ is not locally of weak (1,1) type.
\end{proof}

\end{document}